\documentclass[a4paper,12pt]{article}
\usepackage{latexsym,amssymb,amsmath,amsthm}
\usepackage{url}
\usepackage{amsrefs}
\usepackage{comment}
\newcommand{\Aut}[1]{\mathrm{Aut}({#1})}
\usepackage{longtable}

\DeclareMathOperator\UCA{UCA}
\DeclareMathOperator\CA{CA}
\DeclareMathOperator\UCAN{UCAN}
\DeclareMathOperator\CAN{CAN}

\newtheorem{theorem}{Theorem} 
\newtheorem{conjecture}{Conjecture} 

\newtheorem{corollary}[theorem]{Corollary}

\newcommand{\ignore}[1]{}

\begin{document}

\title{On the Structure of Small Strength-$2$ Covering Arrays}

\author{Janne I. Kokkala\footnote{Supported by the Aalto ELEC Doctoral School, Nokia Foundation, and Academy of Finland, Project \#289002. Present address:  Department of Theoretical Computer
Science, School of Electrical Engineering and Computer Science, KTH
Royal Institute of Technology, SE-100 44 Stockholm, Sweden.} \\
Department of Communications and Networking\\
Aalto University School of Electrical Engineering\\
P.O.\ Box 15400, 00076 Aalto, Finland
\and
Karen Meagher\thanks{Supported in part by an NSERC discovery grant.}\\
Department of Mathematics and Statistics\\
University of Regina\\
Regina, SK, S4S 0A2\\
Canada
\and
Reza Naserasr\footnote{ANR-17-CE40-0022}\\
Institut de Recherche en Informatique Fondamentale \\
B\^{a}timent Sophie Germain\\
8 place Aur\'{e}lie Nemours\\
75013 Paris, France
\and
Kari J. Nurmela\footnote{Present address: Mankkaanmalmi 8 A, 02180 Espoo, Finland}, Patric R. J. \"Osterg\aa rd\footnote{Supported in part by the Academy of Finland, Project \#289002.}\\
Department of Communications and Networking\\
Aalto University School of Electrical Engineering\\
P.O.\ Box 15400, 00076 Aalto, Finland
\and
Brett Stevens\thanks{Supported in part by an NSERC discovery grant.}\\
School of Mathematics and Statistics\\
Carleton University\\
1125 Colonel By Drive\\
Ottawa, ON, K1S 5B6  \\
Canada
}

\date{ }

\maketitle

\begin{abstract}
  A covering array $\CA(N;t,k,v)$ of strength $t$ is an $N \times k$
  array of symbols from an alphabet of size $v$ such that in every
  $N \times t$ subarray, every $t$-tuple occurs in at least one row.
  A covering array is \emph{optimal} if it has the smallest possible
  $N$ for given $t$, $k$, and $v$, and \emph{uniform} if every symbol
  occurs $\lfloor N/v \rfloor$ or $\lceil N/v \rceil$ times in every
  column.  Prior to this paper the only known optimal covering arrays
  for $t=2$ were orthogonal arrays, covering arrays with $v=2$
  constructed from Sperner's Theorem and the Erd\H{o}s-Ko-Rado
  Theorem, and eleven other parameter sets with $v>2$ and $N >
  v^2$. In all these cases, there is a uniform covering array with the
  optimal size. It has been conjectured that there exists a uniform
  covering array of optimal size for all parameters.  In this paper a
  new lower bound as well as structural constraints for small uniform
  strength-$2$ covering arrays are given. Moreover, covering arrays
  with small parameters are studied computationally.  The size of an
  optimal strength-$2$ covering array with $v > 2$ and $N > v^2$ is
  now known for $21$ parameter sets.  Our constructive results
  continue to support the conjecture.
\end{abstract}

\section{Introduction}\label{intro}

A covering array $\CA(N;t,k,v)$ of strength $t$ is an $N \times k$
array of symbols from an alphabet of size $v$ such that in every
$N \times t$ subarray, every $t$-tuple occurs in at least one row.  We
will use $Z_v = \{0,1,\ldots ,v-1\}$ as the alphabet for all of our
covering arrays. A covering array is \emph{optimal} if it has the
smallest possible $N$ for given $t$, $k$, and $v$, and \emph{uniform}
if every symbol occurs either $\lfloor N/v \rfloor$ or $\lceil N/v \rceil$
times in every column. A uniform $\CA(N;t,k,v)$ is denoted by
$\UCA(N;t,k,v)$.  The smallest value of $N$ for which a $\CA(N;t,k,v)$
(respectively $\UCA(N;t,k,v)$) exists is denoted by $\CAN(t,k,v)$ (respectively
$\UCAN(t,k,v)$).

Covering arrays are extensively studied designs with many
applications. There are several surveys of covering arrays
\cite{col0,hart0,law0}; for more recent studies see
\cite{akh0,col3,mr3440525,francetic2016, mr3565209,tor2,tza1}.
Uniform covering arrays are particularly useful since they are used in
some constructions to create larger covering arrays
\cite{colbourn_asymptotic_nodate,francetic2016, sarkar_upper_2017,
  MR3328867}. In this work, we only consider strength-$2$ covering
arrays; thus we omit the parameter $t$ for brevity, and write
$\CA(N;k,v)$ and $\UCA(N;k,v)$ instead of $\CA(N;2,k,v)$ and
$\UCA(N;2,k,v)$, respectively. We also use $\CAN(k,v)$ and $\UCAN(k,v)$
for $\CAN(2,k,v)$ and $\UCAN(2,k,v)$.

A covering array with $N=v^2$ is an orthogonal array and it is
necessarily both uniform and optimal.  For $v=2$ it is known (\cite{katona, kleitmanspencer}) that $\CAN(k,2) = n$, where
\[
{n-2 \choose \lceil (n-1)/2  \rceil} < k \leq {n-1 \choose \lceil n/2  \rceil}.
\]
Moreover, for $v=2$ and all $k$, there is a uniform
covering array of optimal size, so $\UCAN(k,2) = \CAN(k,2)$ (this is a
consequence of a graph homomorphism and graph core
result~\cite[Theorem 5]{mea0}).

Prior to this work, as far as we have been able to verify, the other optimal values known when $N > v^2$ were
$\CAN(5,3) = 11$, $\CAN(6,3) = \CAN(7,3) = 12$,
$\CAN(8,3) = \CAN(9,3) = 13$, $\CAN(10,3) = 14$,
$\CAN(6,4) = 19$, $\CAN(7,4) = 21$, $\CAN(7,5) = 29$, $\CAN(4,6) = 37$, and $\CAN(5,6) = 39$
(see Table~\ref{tab:res} in the current paper for references).  In all these cases, there exists an optimal covering array that is also uniform. In fact, to date there has not been a single set of parameters found for which none of the optimal covering
arrays is uniform. This has led the second and sixth author of the
current paper to make the following conjecture \cite[Conjecture
1]{mea0}.

\begin{conjecture} \label{conj:uca}
If there exists a $\CA(N;k,v)$ then there exists a $\UCA(N;k,v)$. 
\end{conjecture}

Recently, Torres-Jimenez~\cite{tor1} found examples of optimal, but
not uniform, covering arrays with the additional property that the
array has the maximum number of columns (maximum $k$) for the given
number of rows (given $N$).  One generalization of covering arrays is
covering arrays avoiding forbidden edges where certain pairs of
symbols in certain columns are forbidden
\cite{danziger_covering_2009}. There exists an arc-transitive
4-partite graph where the unique optimal covering array avoiding the
edges of the graph cannot be uniform \cite{set_nonu}.  This does not
refute the conjecture but it does show that placing even highly
symmetric constraints on covering arrays can force non-uniformity of
optimal arrays.

An analogous problem has also been studied for covering and packing (error-correcting) codes. For binary covering codes, there are sets of parameters for which all optimal codes are nonuniform \cite{ost0}. For binary error-correcting codes, there are even sets of parameters for which all optimal codes have a nonuniform distribution of coordinate values in all coordinates \cite{ost}.

The main challenge in studying Conjecture~\ref{conj:uca}---in searching
for a counter\-example---is to determine $\CAN(k,v)$. This can be done via a
lower bound and a constructive upper bound that meet. 
There are some well-known constructions for covering arrays.
Specific covering arrays can be found by metaheuristic search
techniques~\cite{coh1, nur}, using constraint programming
models~\cite{MR2224851} or by applying post-optimization techniques to
known constructions~\cite{MR2974273}. In practice, however,
strong enough bounds are in general available only for limited sets of
parameters \cites{ste,MR2224851}

A new lower bound on the size of covering arrays is proved in this paper.
Analytical methods can be augmented with computational techniques, which
will be utilized in the current work to determine $\CAN(k,v)$
up to the limits set by the available algorithms and computational resources.

In this paper, new lower bounds and structural constraints on uniform
covering arrays are given in Section~\ref{mathSec}. Computational
methods, including exhaustive search and classification procedures,
are described in Section~\ref{procSec}. Equivalent covering arrays can
be made from permuting columns, rows, or symbols. A central aspect of efficient
exhaustive search in general and classification in particular is avoiding
finding different copies of the same array; this process is called {\em isomorph rejection}.
These concepts are discussed in more detail in Section~\ref{procSec}. An extensive
table of classification results is given in Section~\ref{comp_res}.
Finally, our computational results are discussed in
Section~\ref{resultsSec}, which also contains updated tables of bounds
on $\CAN(k,v)$ and $\UCAN(k,v)$ for $4 \leq k \leq 10$ and
$3 \leq v \leq 6$.

\section{Bounds for small covering arrays} \label{mathSec}

\subsection{A lower bound for uniform covering arrays}

The following theorem can be used to get a lower bound on the size
of a uniform covering array.

\begin{theorem}\label{main_bound_thm}
Let $C$ be a $\UCA(N;k,v)$.  Let $d = \lfloor N/v \rfloor$ and $i = N-vd$. Then 
\[
(k^2-3k+2v)N^2 - v(k(2v-1)-2)(k-1)N + k(k(v^4-v^3+vi-i^2)  - (v^4 - v^3 + 3vi - 3i^2) ) \geq 0  
  \]
and a necessary condition for equality is that every pair of rows in $C$ agree in at least one and at most two columns.
\end{theorem}
\begin{proof}
Let $C$ be a $\UCA(N;k,v)$, $d = \lfloor N/v \rfloor$, and $i = N-vd$.
To arrive at the inequality, we will find an upper and a lower bound
for the number of pairs of rows which agree in at least one
position. An upper bound is the total number of pairs of rows, ${N
  \choose 2}$.  

To get a lower bound on the number of pairs of rows we introduce two
new parameters. Define $M_1$ to be the number of triples $(r,r',c)$
for which rows $r$ and $r'$ agree in column $c$. Further define $M_2$
to be the number of quadruples $(r,r',c,c')$ for which rows $r$ and
$r'$ agree in columns $c$ and $c'$.  Then $M_1 -M_2$ is a lower bound
on the number of pairs of rows which agree in at least one position.
(Indeed, $M_1 -M_2$ consists of the first two terms in the summation
using the principle of inclusion and exclusion.)

This gives us the bound
\begin{equation}
M_1 - M_2 \leq \binom{N}{2}, \label{eq:m1m2bound}
\end{equation}
which is tight if and only if every pair of rows in $C$ agree in at least one and at most two columns.

Since the array is uniform, in every column there are $i$ symbols
which appear $d+1$ times and $v-i$ symbols which appear only $d$
times.  Thus the contribution to $M_1$ from any column is
$i{d+1 \choose 2} + (v-i){d \choose 2}$ and the sum of these over all
columns is
\begin{equation}
M_1 = k \left( i{d+1 \choose 2} + (v-i){d \choose 2} \right) = \frac{kd(N-v+i)}{2}. \label{eq:m1}
\end{equation}

Next, we find an upper bound for $M_2$. Consider columns $c$ and $c'$. 
Let $\lambda_{x,y}^{c,c'}$ be the number of rows $r$ such that
$C_{r,c}=x$ and $C_{r,c'}=y$, and let $\mu_x^{c,c'}$ be the number of
pairs of rows $r$ and $r'$ such that $C_{r,c}=C_{r',c}=x$ and
$C_{r,c'}=C_{r',c'}$. It follows from the definition that 
\begin{equation}
\mu_{x}^{c,c'}=\sum_y \binom{\lambda_{x,y}^{c,c'}}{2}. \label{eq:mux_lambdaxy}
\end{equation}
The number of pairs of rows that agree in columns $c$ and $c'$ is then $\sum_x \mu_x^{c,c'}$.

For each $x$, let $m_x^{c}$ be the number of times $x$ occurs in
column $c$. Since $\sum_y \lambda_{x,y}^{c,c'} = m_x^c$ and
$\lambda_{x,y}^{c,c'} \geq 1$ for all $y$, it can be seen that
Equation~\eqref{eq:mux_lambdaxy} is maximized for each $x$ when there
is a $y_x$ such that $\lambda_{x,y_x}^{c,c'} = m_x^c + 1-v$, and
$\lambda_{x,y}^{c,c'}=1$ for all $y \neq y_x$. This gives
\begin{equation}
\sum_x \mu_{x}^{c,c'} \leq \sum_x \binom{m_x^c + 1-v}{2} = i \binom{d+2-v}{2} + (v-i) \binom{d+1-v}{2}. \label{eq:mux}
\end{equation}
The last equality follows from the fact that $i$ symbols occur $d+1$
times in column $c$, and $v-i$ symbols occur $d$ times in column
$c$. This bound is attained if and only if there is a permutation
$\pi$ of $\{0,1,\dots,v-1\}$ such that the number of times $x$ appears
in column $c$ equals the number of times $\pi(x)$ appears in column
$c'$ for each $x$, and $\lambda_{x,y} = 1$ whenever $y \neq \pi(x)$.

Summing \eqref{eq:mux} over all pairs of columns gives us an upper bound for $M_2$, 
\begin{equation}
M_2 = \sum_{c,c'} \sum_x \mu_{x}^{c,c'} \leq \frac{k(k-1)(d+1-v)(N-v^2 + i)}{4}.  \label{eq:m2}
\end{equation}

Applying \eqref{eq:m1} and \eqref{eq:m2} to \eqref{eq:m1m2bound} and
multiplying both sides by $4v$ yields the bound from the theorem.
\end{proof}

Theorem~\ref{main_bound_thm} is particularly useful for small $k$.

\begin{corollary}\label{v+2_cor}
  If there exists a $\UCA(N;v+2,v)$, then $N \geq v^2 + v -1$.
  Further, if $N=v^2+v-1$, then every pair of rows must agree in
  either one or two positions, and in each pair of
  columns there are exactly $v-1$ disjoint pairs of symbols that
  appear twice.
\end{corollary}

Similarly we can apply Theorem~\ref{main_bound_thm} to covering arrays with few columns.
\begin{corollary}\label{v+3_cor}
Assume that there exists a $\UCA(N;v+j,v)$.
\begin{enumerate}
\item If $j=3$, then $N > v^2 + 3v/2 -5/2$. If additionally $2 < v \leq 11$, then $N > v^2 + 3v/2 -2$. 
\item If $j=4$ then, $N > v^2 + 2v -5$.  If additionally $v \leq 6$, then $N \geq v^2 + 2v - 4$. 
\item If $j=5$, then $N > v^2 + 7v/3 - 13/2$.  
\item If $j=6$, then $N > v^2 + 8v/3 - 21/2$.  
\item If $j=7$, then $N > v^2 + 3v - 15$.
  \end{enumerate}
\end{corollary}

In the previous corollary, there are similar improvements possible in the constant term in the lower
bound on $N$ when $v$ is sufficiently small for $j \geq 5$. We only
state the improved bounds for $j=3$ and 4.

The form of the bound in Theorem~\ref{main_bound_thm} does not let us
easily identify its behaviour as a function of $k$, but, by losing the
accuracy given by the residue of $N \bmod v$, we can obtain a weaker
bound that has a more directly computable form.

\begin{corollary}\label{simpler_cor}
  Let
  \begin{align*}
    b &= (2v-3)k^2 + (-2v+5)k +(-4v + 2),\\
    a &= 2k^2 - 6k + 4v, \mbox{ and}\\
    D &= k^4 +(8v^2-16v+2)k^3 + (-8v^3+24v-3)k^2 +(8v^3 -8v^2 -8v -4)k  + 4.
    \end{align*}
If there exists a $\UCA(N;k,v)$, then 
{\begin{equation}\label{simple_ineq}
N \geq v\left(\frac{b + \sqrt{D }}{a}\right )
\end{equation}}
\end{corollary}
\begin{proof}
  Consider a uniform covering array $\CA(N;k,v)$. Let $N'$ be the
  least multiple of $v$ that is at least $N$. This means that
  $N' \leq N+v-1$ and the uniform $\CA(N;k,v)$ can be extended
  to a $\CA(N';k,v)$ in which each column has each symbol occurring
  exactly $N'/k$ times.  Theorem~\ref{main_bound_thm} can be applied
  to the $\CA(N';k,v)$, to get that
\[
2vN'(N'-1) \geq k\left ( 2(N')(N'-v) - (k-1)(N'-v^2+v)(N'-v^2)  \right ).
\]
This reduces to 
\[
0 \leq (k^2-3k+2v)(N')^2 - v(k-1)(2kv-k-2)N' - k(k-1)(v^4-v^3) \leq 0
\]
Then $N'$ must be bounded below by the quadratic's larger root.
The result follows since  $N' -(v-1)\leq N$.
\end{proof}


By taking the derivative of Inequality~(\ref{simple_ineq}) with respect
to $k$ and approximating its roots we compute that this bound reaches
it maximum at a value of $k$ less than, but close to
\[
k_{\max} \approx \frac{16v^2-20v-15}{8v-16} = 2v + \frac{3}{2} + \frac{9}{8v-16}.
\]
The error in this approximation is less then $0.5$ after $v=16$. The value of the bound at this maximum point is approximately
\[
N_v = \UCAN(2,k_{\max},v) \geq 2.4142v - 1.17678 - \frac{3.5026v + 10.2260}{8v^2+8v-4}.
\]
That is, for $k > k_{\max}$ the value of the bound from
Corollary~\ref{simpler_cor} is smaller than $N_v$.  Since we know that
$\UCAN(2,k+1,v) \geq \UCAN(2,k,v)$, the bound from
Corollary~\ref{simpler_cor} loses its utility for any $k > k_{\max}$.
The maximum useful $k$ for the bound of Theorem~\ref{main_bound_thm}
must also be close to this $k_{\max}$. In our classification results
six uniform covering arrays meet the bound from
Theorem~\ref{main_bound_thm}.  Five have $k=v+2$ and one,
$\UCA(21;7,4)$, has $k = v+3$.

\subsection{Constraints on covering arrays with $v+2$ columns}

The strongest structural conditions implied by equality in
Theorem~\ref{main_bound_thm} happen when $k=v+2$ and $N = v^2 + v
-1$. We further investigate covering arrays with these parameters.
First we introduce some notation. In a uniform covering array
$\UCA(v^2+v-1; k,v)$, in each column every symbol occurs either $v$
times or $v+1$ times. An entry in a $\UCA(N;k,v)$ is called a
\textsl{high frequency entry} if the symbol in the entry occurs at
least $v+1$ times in the entry's column.

\begin{theorem} 
  Let $C$ be a $\UCA(v^2+v-1; v+2,v)$ and let $a_i$ be the number of
  rows that contain exactly $i$ high frequency entries. Then
\begin{align}
\sum_{i=0}^{v+2} a_i &= v^2 + v -1, \label{ai1}\\
\sum_{i=0}^{v+2} ia_i &= (v+2)(v-1)(v+1) , \label{ai2}\\  
\sum_{i=0}^{v+2} i^2a_i &=  (v+2)(v+1)^2(v-1). \label{ai3} 
\end{align}
Further,  $a_0 \leq 1$, and $a_1=a_2= 0$.
\end{theorem}

\begin{proof}
  Let $C$ be a $\UCA(v^2+v-1; v+2,v)$. For a column, $c$, denote by
  $S_c$ the set of symbols in high frequency entries.  We know
  $|S_c| = v-1$.

Equation~(\ref{ai1}) is established by simply counting the rows of
  $C$. Equation~(\ref{ai2}) is established by computing the
  cardinality of the set
\[
  \{(x,c,r)\; |\; 0 \leq c < v+2, \; C_{r,c}=x \in S_c \}.
\]
Equation~(\ref{ai3}) is established by computing the cardinality of the set
\[
  \{(x,c,y,c',r)\; |\; 0 \leq c\neq c' < v+2, \; C_{r,c}=x \in S_c, C_{r,c'} = y \in S_{c'}\}.
\]

There is exactly one symbol per column that is repeated exactly
$v$ times. So if two rows had no high frequency entries, then both
rows would only contain the symbols that occur exactly $v$ times. This
would mean that a pair of such symbols is repeated and $C$ could not
be a covering array. Thus $a_0 \leq 1$. 

To establish that $a_1 = a_2 =
0$, let $r$ be a fixed row containing $i$ symbols which appear $v+1$
times in their column. For any of the other $v^2+v-2$ rows $b$, let
$\mu_{r,b}$ be the number of columns where rows $r$ and $b$ agree.
Counting the flags $(c,b)$ with $0 \leq c < v+2$, $b \neq r$ and
$C_{b,c} = C_{r,c}$ we have 
\begin{align}
\sum_{b \neq r} \mu_{r,b} & = v^2 + v -2 + i,  \label{ai4} \\ 
\overline{\mu} &= \frac{v^2 +v -2 +i}{v^2 + v -2} .
\end{align}
Counting the flags $(c,c',b)$ such that $0 \leq c \neq c' < v+2$, $C_{b,c} = C_{r,c}$, and $C_{b,c'} = C_{r,c'}$ we have
\begin{equation}
\sum_{b \neq r} \binom{\mu_{r,b}}{2} \leq \binom{i}{2},
\end{equation}
(this follows since only high frequency entries can occur twice).
Using Equation~(\ref{ai4}), this implies that
\begin{equation}
\sum_{b \neq r} \mu_{r,b}^2 \leq v^2 + v -2 + i^2 . \label{ai5} \\ 
\end{equation}
Now we get
\begin{align*}
  0 & \leq \sum_{b \neq r} (\mu_{r,b} - \overline{\mu})^2 \\
  &= -\overline{\mu}^2(v^2+v-2) + \sum_{b \neq r} \mu_{r,b}^2 \\
    &\leq i\frac{i(v^2+v-3)-(2v^2+2v-4)}{v^2+v-2}.
\end{align*}
Which implies that $i=0$ or $i \geq \lceil 2 + 2/(v^2+v-3) \rceil = 3$.
\end{proof}

\begin{corollary} 
Let $C$ be a $\UCA(v^2+v-1; v+2,v)$ and let $a_i$ be the number of rows
that have exactly $i$ high frequency entries.
If $a_0 = 1$, then $a_{v+1} = v^2 + v - 2$ and $a_i = 0$ for all other $i$.
\end{corollary}
\begin{proof}
If $a_0 = 1$ then \eqref{ai1},~\eqref{ai2} and~\eqref{ai3} imply
that the average $i$ is $\overline{i} = v+1$. The variance in the
distribution of $i \neq 0$ is equal to 0, since
\begin{align*}
\sum_{i=1}^{v+2} (i-\overline{i})^2 &= -\overline{i}^2(v+2)(v-1) + \sum_{i=1}^{v+2} i^2a_i \\
&= (v+2)(v-1)(v+1)^2 - (v+2)(v-1)(v+1)^2.
\end{align*}
\end{proof}

All this leaves open the possible existence of a $\CA(N;v+2,v)$ with
$N < v^2 + v -1$ (and thus smaller than those described in
Corollary~\ref{v+3_cor}) if only the covering array is {\em not} uniform.  However
some constraints exist even in this case for $\CA(N;v+2,v)$. In
\cite{ste}, the following result is proved using a similar counting
method.

\begin{theorem} 
  Assume that there exists a $\CA(N;k,v)$ that has a row that contains
  at most two high frequency entries.  If $k=v+2$, then
  $N \geq v^2 + v -1$; and if $k \geq v+3$, then $N \geq v^2+v$.

  Assume that there exists a $\CA(N;k,v)$ that has a row that contains
  at most three high frequency entries. If $k \geq v+2$, then
  $N \geq v^2+v-1$.
\end{theorem}

\section{Classifying covering arrays}\label{procSec}

In all of our computer-aided studies of covering
arrays, we fix the parameters of the array: the order, $v$; the degree, $k$;
and the size, $N$. We further consider covering arrays as multisets of their rows.
(Given two multisets, $S$ and $T$, the 
\emph{multiset sum} $S \uplus T$ is the set for which the multiplicity of each 
element is the sum of its multiplicities in $S$ and $T$.)
Two covering arrays are then said to be \emph{equivalent} if one can be
obtained from the other by a permutation of the columns and
by column-wise permutations of the elements of $Z_v$.
A transformation that maps a covering array $C$ onto
itself is an \emph{automorphism}, and the set of all 
automorphisms form the \emph{(full) automorphism
group} of $C$, denoted by $\Aut{C}$.
Our computer search builds all inequivalent covering arrays with a
given parameter set.

For isomorph rejection, we represent  
covering arrays as coloured graphs (this is described below) and, using
\emph{nauty}~\cite{mck_pip}, we may also determine the automorphism group
for graphs and thereby the corresponding arrays.
The coloured graph $G$ corresponding to a covering array $C$ is
constructed in the following standard way~\cite{OBK99}.

The vertices of $G$ will be coloured with two colours. Colourings of graphs are
the mechanism that \emph{nauty}~\cite{mck_pip} uses to forbid certain mappings between
vertices; they need not be proper colourings.
First, $G$ contains $k$ disjoint copies of the
complete graph of order $v$; all of these vertices are coloured with
the first colour. These vertices represent the entries in the columns of the $\CA$;
the $i$th copy of the complete graph corresponds to the $i$th column
and the $j$th vertex in each copy corresponds to the symbol $j$ in
that column.  Further, $G$ contains $N$ vertices, representing the rows of $C$,
all coloured with the second colour. Each such vertex
is connected to the $k$ vertices corresponding to the column-symbol
pairs that occur in that row. Note that the obvious homomorphism
$\Aut{G} \rightarrow \Aut{C}$ has a nontrivial kernel if there are
duplicate rows.

Our main method, presented in Section~\ref{sec:comp_main}, constructs
one representative from each equivalence class of covering arrays, $\CA(N;k,v)$. 
This is done by starting with a set of
representatives of the equivalence classes of 
 covering arrays $\CA(N;2,v)$, and sequentially adding columns, rejecting
equivalent covering arrays after every step. Canonical augmentation
\cite[Sect.\ 4.2.3]{kas}, \cite{mck2} is used when extending
representatives of covering arrays $\CA(N;k',v)$ to representatives of
covering arrays $\CA(N;k'+1,v)$; this part is described in detail in
Section~\ref{sec:comp_isomorph}.  Further, since our goal is to
classify all $\CA(N;k,v)$ for certain $k$ but not necessarily those
that have a smaller number of columns, we can occasionally speed up
the search by rejecting some partial arrays that cannot be extended to
a full $k$-column covering array; the method is described in
Section~\ref{sec:comp_force}.  When studying only uniform covering
arrays, it is easy to modify the algorithm to require uniformity.

\subsection{Algorithm} \label{sec:comp_main}

In this section, we describe the algorithm for classifying all
covering arrays $\CA(N; k+1, v)$, starting from a set of
equivalence class representatives of covering arrays $\CA(N; k, v)$.  
To apply such an algorithm, we need a base case,
which here is a classification of the covering
arrays $\CA(N;2, v)$. Since all $v^2$ pairs of symbols must occur in the two columns
of those covering arrays, we may focus on the $N-v^2$ excess rows and
just the equivalence issue. For the excess part in the first column,
the symbol distributions are in one-to-one correspondence to the
integer partitions of $N-v^2$ into at most $v$ parts.  In the second
column we may take obvious symmetries into account to reduce the
number of candidates considered. Finally, equivalent arrays are
rejected.

\subsubsection{Extending covering arrays}

Consider a covering array $C'$ obtained by adding a column to a covering array $C$. 
The symbols in the new column in $C'$ induce a partition of the rows of $C$ into covering
arrays of strength $1$. We call a subset of $C$ that is a covering array of strength $1$ 
a \emph{cover} of $C$. If no proper subset of a cover of $C$ is a cover of $C$, we call 
it a \emph{minimal cover} of $C$. Each cover of $C$ has one or more subsets that are 
minimal covers. For a cover $D$, we denote the lexicographically smallest subset that 
is a minimal cover by $\phi(D)$.

When extending covering arrays, we first determine 
$\mathcal{D}$, the set of all minimal covers of $C$. Then we find all sets 
$\{D_1,D_2,\dots,D_v\}$ of $v$ minimal covers that pack inside $C$, that is, 
$\biguplus_i D_i \subseteq C$. For each such set, we generate all full partitions 
$\{C_1,C_2,\dots,C_v\}$ of $C$, where $D_i \subseteq C_i$ for all $i$, by adding the 
remaining rows in the sets $D_i$ in all possible ways. To avoid repetition, we reject 
in the search all partitions for which $D_i \neq \phi(C_i)$. 
To get a covering array from an unlabeled partition, we map the symbols to the parts 
such that the resulting covering array is lexicographically smallest; this 
mapping from partitions to covering arrays is required in the sequel.

\subsubsection{Isomorph rejection} \label{sec:comp_isomorph}

Having generated all extensions of $C$ up to permutation of symbols in
the last column, canonical augmentation is used for isomorph rejection
in two phases.  The first phase rejects some arrays and ensures that
two remaining arrays can be equivalent only if they were generated
from the same $C$, and further that there is an automorphism of $C$
that maps one onto another. The second phase then accepts precisely
one array from each equivalence class. Actually the two phases can be
carried out in arbitrary order, and in our implementation Condition~\ref{enum:chi} below is checked first to help in validating
the results, to be discussed later.

In the first phase, we use the $v$-tuple consisting of the counts of
each symbol in that column sorted in descending order as an
invariant of a column. For example, if a column contains three
entries equal to $0$, six entries equal to $1$, and three entries
equal to $2$, then the invariant is $(6,3,3)$.  A covering array $C'$
passes the first phase of canonical augmentation if:
\begin{enumerate}
\item no other column has lexicographically smaller invariant than the last column, and \label{enum:invariant}
\item out of those columns with the same invariant as the last column, the last column is in the orbit 
that gets the smallest label in a canonical labeling by \emph{nauty}.
\end{enumerate}
Let $\mu$ be the largest multiplicity of a symbol in the column of $C$ that has the 
smallest invariant. The first condition ensures that in a canonically augmented 
$C'$, there is no symbol in the new column with multiplicity larger than $\mu$. 
This allows us to remove from $\mathcal{D}$ all minimal covers with size larger 
than $\mu$ before the search begins and also not consider full partitions of $C$ 
for which one part has size larger than $\mu$.

For the second phase, we treat the array $C'$ as a partition of $C$. Let $c$ be an arbitrarily chosen row of $C$ which has multiplicity 1 and is held fixed for the search of extensions of $C$. Let $\mathcal{C}$ be the orbit of $C'$ under the action of $\Aut{C}$, and let $\chi(C',c) = \phi(A)$ where $A$ is the part in $C'$ that contains $c$. 
An array $C'$ passes the second phase if 
\begin{enumerate}
\setcounter{enumi}{2}
\item $\chi(C',c) \leq \chi(C'',c)$ for all $C'' \in \mathcal{C}$, and \label{enum:chi}
\item $C'$ is the smallest in the set $\{C'' \in \mathcal{C} : \chi(C'',c) = \chi(C',c) \}$, in terms of lexicographical ordering of the corresponding arrays.
\end{enumerate}
Condition~\ref{enum:chi} allows us to reject from $\mathcal{D}$ all minimal covers 
$D$ that contain $c$ for which there is a $g \in \Aut{C}$ such that $c \in gD$ 
and $gD < D$. To this end, the row $c$ is selected to be the one that maximizes the
number of minimal covers that are rejected from $\mathcal{D}$.

\subsubsection{A pruning condition}\label{sec:comp_force}

Let $N$ and $v$ be integers with $N<v(v+1)$. Let $C$ be a $\CA(N;k,v)$ and let $C'$ be a $\CA(N; k', v)$ that is obtained by adding $\delta = k'-k$ columns to $C$.
Since $N$ is strictly less than $v^2+v$, each of the last $\delta$ columns of $C'$ contain at least one 
symbol of multiplicity $v$, each of which corresponds to a cover of size $v$ in $C$. 
For each pair of columns and each symbol of multiplicity $v$, the two covers 
intersect in exactly one row (considering duplicate rows as separate elements).  
Thus $C$ has a set of $\delta$ covers of size $v$ which pairwise intersect 
in only one row. 

If we are interested only in covering arrays $\CA(N;k',v)$ and not 
in covering arrays $\CA(N;k'',v)$ for any
$k < k'' < k'$, we can restrict our search to the covering arrays
$\CA(N;k,v)$ that satisfy this property. We gain further speedup by
running the search for each possible way to fix the set of $\delta$
covers of size $v$ that intersect in the desired way, as fixing the
set allows rejecting many covers in $\mathcal{D}$ immediately.

\subsubsection{Some implementation details}

A core subroutine of the algorithm is that of 
finding subsets of $\mathcal{D}$ that pack inside $C$. This was
implemented in two different ways, one using \emph{Cliquer} \cite{NO03} and 
one using \emph{libexact} \cite{kas2}. The simpler approach using \emph{Cliquer} 
is faster in some cases, but in most cases, the approach using \emph{libexact} is faster. 

To use \emph{Cliquer} we define $G$ to be a graph with a vertex for
every cover in $\mathcal{D}$ and an edge between two covers if their
multiset sum is a subset of $C$.  A packing corresponds to a clique of
size $v$ in $G$, but a clique may not be a valid packing if there are
not enough duplicate rows in $C$. Further, when $N\geq v(v+1)$,
two covers in a packing may be identical, so elements in $\mathcal{D}$
for which all rows have multiplicity greater than $1$ in $C$ must be
represented with duplicated vertices in $G$.

The library \emph{libexact} is used to find all solutions to a system
of linear equations $Ax=b$ with $0 \leq x_j \leq u_j$ where $A$ is a
$(0,1)$-matrix. We set up the instance as follows. For each cover in
$\mathcal{D}$, we have a variable whose value is the multiplicity of the
cover in the packing. For each different row in $C$, we then have an
inequality; namely, the row should occur in the packing at most as many
times as it occurs in $C$. To encode this as an equality, we add a
variable for each row that whose value is the slack in that inequality (that
is, how many instances of the row in $C$ are not covered by the
packing). Further, to force a solution to have exactly $v$ covers, we
add a condition that the sum of variables corresponding to covers in
$\mathcal{D}$ must be equal to $v$. The upper bounds of each variable are directly
obtained from the equalities, as all variables are nonnegative.

We introduce further slack variables to account for conditions on the sizes of covers in the packing. 
These slack variables have no effect on the solutions but they speed up the search by identifying 
some branches that cannot lead to a solution. For a valid packing, let $M_v$ 
be the number of covers of size $v$, let $M_{v+1}$ be the number of covers of size $v+1$, 
and let $M_{\geq v+2}$ be the number of covers of size at least $v+2$. We have
\begin{gather}
M_v  \leq v, \label{eq:ineq_slack1} \\
v M_v + (v+1) M_{v+1} + (v+2)M_{\geq v+2}  \leq N. \label{eq:ineq_slack2}
\end{gather}
Here \eqref{eq:ineq_slack2} is obtained by counting the rows in each
cover. We define $s_1$ and $s_2$ to be the slack variables in
\eqref{eq:ineq_slack1} and \eqref{eq:ineq_slack2}, respectively,
giving
\begin{gather}
  s_1 + M_v  = v, \label{eq:value_s1} \\
  s_1 + s_2 + M_{\geq v+2} = N - v^2, \label{eq:value_s2}
\end{gather}
where we used $M_v + M_{v+1} + M_{\geq v+2} = v$ to get
\eqref{eq:value_s2}. These equations can be directly implemented by
writing $M_v$ and $M_{\geq v+2}$ as sums of the variables
corresponding to covers of size $v$ or at least $v+2$,
respectively.
The upper bound of $s_1$ and $s_2$ in the \emph{libexact} instance is set to $N-v^2$, which follows from \eqref{eq:value_s2}.

\subsection{Computational results}\label{comp_res}

A classification is here carried out for $v=3,4,5,6$ and values of $N$ up
to the computational limit. The listed times refer to a single logical core
of an Intel Xeon E5 family processor with multi-threading enabled. 
Specifically, $\CA(N;k,v)$ are classified
for $10 \leq N \leq 14$ when $v=3$, for $17 \leq N \leq 20$ when
$v=4$, for $26 \leq N \leq 29$ when $v=5$, and for $37 \leq N \leq 40$
when $v=6$.  The full classification is performed for all possible
values of $k$, except for the cases of $\CA(29; k,5)$ and
$\CA(40; k,6)$, where some values of $k$ were skipped using the method
described in Section~\ref{sec:comp_force} to get to the cases of
$\CA(29; 7,5)$ and $\CA(40; 6,6)$; the latter has 0 solutions so
a $\CA(40; k,6)$ exists exactly when $k \leq 5$. 

Due to the computational time, we are unable to carry out a complete
classification in the following cases: $\CA(21; k, 4)$ with
$k \in \{3,\ldots,8\}$; $\CA(30; k,5)$ with $k \in \{3,\ldots,8\}$; and
$\CA(41; k,6)$ with $k \in \{3,\ldots,7\}$. For example, we predict that classifying
$\CA(21,7,4)$ would take 130 core-years.

In all these cases, the number of uniform arrays is also
obtained. Finally, in the uniform cases, the classification of
$\UCA(21; k,4)$ and $\UCA(30; k,5)$ is performed exhaustively and for
$\UCA(41; k,6)$ partially, skipping levels to get to $\UCA(41; 7,6)$;
the latter has 0 solutions so a $\UCA(41;k,6)$ exists exactly when
$k \leq 6$.

A complete table of results obtained for $\CA(N;k,v)$ and
$\UCA(N;k,v)$ is given in Table~\ref{tab:comp_res_detailed}. When
$\delta$ is not given, all covering arrays and uniform covering arrays
are classified. When $\delta$ is given, the stated quantities are the
numbers of covering arrays or uniform covering arrays obtained using
the method in Section~\ref{sec:comp_force} with the given
$\delta$. These quantities are lower bounds for the numbers of all
covering arrays and uniform covering arrays. In cases where the count
of all covering arrays is not given, only uniform covering arrays are
classified. \emph{Cliquer} is used in the cases marked with \textdagger, 
and \emph{libexact} is used in all other cases.

Because a covering array
occurs as a subset of a covering array with more rows, the
classification results of smaller $N$ could be obtained from the
results of larger $N$; however, the running times are reported
separately to give an idea of how the running time of the algorithm
depends on $N$. The method to generate inequivalent $2$-column arrays
is not optimized and the time is not comparable to the other times so
the time for $k=2$ is not reported; in all cases the generation takes
less than $10$ seconds.

\begin{longtable}{rrrr|rrrl}
\caption{Detailed computational results} \\
\label{tab:comp_res_detailed}
$v$ & $N$ & $k$ & $\delta$ & \# CA & \# UCA & CPU time & \\
 \hline 
\endfirsthead
\caption[]{Detailed computational results (cont.)} \\
$v$ & $N$ & $k$ & $\delta$ & \# CA & \# UCA & CPU time & \\
 \hline 
\endhead

$3$ & $10$ & $2$ & $$ & $1$ & $1$ &  &   \\
$3$ & $10$ & $3$ & $$ & $3$ & $3$ & $< 0.01$ s &   \\
$3$ & $10$ & $4$ & $$ & $2$ & $2$ & $< 0.01$ s &   \\
$3$ & $10$ & $5$ & $$ & $0$ & $0$ & $< 0.01$ s &   \\
\hline
$3$ & $11$ & $2$ & $$ & $3$ & $1$ &  &   \\
$3$ & $11$ & $3$ & $$ & $20$ & $9$ & $0.01$ s &   \\
$3$ & $11$ & $4$ & $$ & $27$ & $8$ & $0.02$ s &   \\
$3$ & $11$ & $5$ & $$ & $3$ & $3$ & $0.01$ s &   \\
$3$ & $11$ & $6$ & $$ & $0$ & $0$ & $< 0.01$ s &   \\
\hline
$3$ & $12$ & $2$ & $$ & $7$ & $1$ &  &   \\
$3$ & $12$ & $3$ & $$ & $134$ & $9$ & $0.02$ s &   \\
$3$ & $12$ & $4$ & $$ & $987$ & $53$ & $0.16$ s &   \\
$3$ & $12$ & $5$ & $$ & $891$ & $125$ & $0.38$ s &   \\
$3$ & $12$ & $6$ & $$ & $13$ & $11$ & $0.10$ s &   \\
$3$ & $12$ & $7$ & $$ & $1$ & $1$ & $< 0.01$ s &   \\
$3$ & $12$ & $8$ & $$ & $0$ & $0$ & $< 0.01$ s &   \\
\hline
$3$ & $13$ & $2$ & $$ & $16$ & $3$ &  &   \\
$3$ & $13$ & $3$ & $$ & $937$ & $151$ & $0.09$ s &   \\
$3$ & $13$ & $4$ & $$ & $53\,523$ & $12\,747$ & $6.0$ s &   \\
$3$ & $13$ & $5$ & $$ & $739\,845$ & $302\,524$ & $144.5$ s &   \\
$3$ & $13$ & $6$ & $$ & $752\,165$ & $506\,680$ & $940.8$ s &   \\
$3$ & $13$ & $7$ & $$ & $24\,934$ & $22\,539$ & $600.9$ s &   \\
$3$ & $13$ & $8$ & $$ & $5$ & $5$ & $11.4$ s &   \\
$3$ & $13$ & $9$ & $$ & $4$ & $4$ & $< 0.01$ s &   \\
$3$ & $13$ & $10$ & $$ & $0$ & $0$ & $< 0.01$ s &   \\
\hline
$3$ & $14$ & $2$ & $$ & $32$ & $4$ &  &   \\
$3$ & $14$ & $3$ & $$ & $5\,973$ & $476$ & $0.51$ s &   \\
$3$ & $14$ & $4$ & $$ & $2\,212\,568$ & $214\,630$ & $580.2$ s &   \\
$3$ & $14$ & $5$ & $$ & $325\,046\,812$ & $43\,473\,308$ & $29.9$ h &   \\
$3$ & $14$ & $6$ & $$ & $7\,759\,008\,032$ & $1\,516\,020\,148$ & $54.1$ d &   \\
$3$ & $14$ & $7$ & $$ & $18\,844\,482\,204$ & $5\,827\,703\,442$ & $446.8$ d &   \\
$3$ & $14$ & $8$ & $$ & $2\,790\,300\,754$ & $1\,429\,724\,866$ & $519.7$ d &   \\
$3$ & $14$ & $9$ & $$ & $17\,068\,936$ & $12\,725\,845$ & $43.4$ d &   \\
$3$ & $14$ & $10$ & $$ & $4\,490$ & $4\,117$ & $4.1$ h &   \\
$3$ & $14$ & $11$ & $$ & $0$ & $0$ & $3.3$ s &   \\
\hline
$4$ & $17$ & $2$ & $$ & $1$ & $1$ &  &   \\
$4$ & $17$ & $3$ & $$ & $6$ & $6$ & $0.03$ s &   \\
$4$ & $17$ & $4$ & $$ & $3$ & $3$ & $< 0.01$ s &   \\
$4$ & $17$ & $5$ & $$ & $4$ & $4$ & $< 0.01$ s &   \\
$4$ & $17$ & $6$ & $$ & $0$ & $0$ & $< 0.01$ s &   \\
\hline
$4$ & $18$ & $2$ & $$ & $3$ & $1$ &  &   \\
$4$ & $18$ & $3$ & $$ & $79$ & $42$ & $0.08$ s &   \\
$4$ & $18$ & $4$ & $$ & $79$ & $31$ & $0.13$ s &   \\
$4$ & $18$ & $5$ & $$ & $201$ & $67$ & $0.08$ s &   \\
$4$ & $18$ & $6$ & $$ & $0$ & $0$ & $0.04$ s &   \\
\hline
$4$ & $19$ & $2$ & $$ & $7$ & $1$ &  &   \\
$4$ & $19$ & $3$ & $$ & $1\,365$ & $191$ & $0.51$ s &   \\
$4$ & $19$ & $4$ & $$ & $12\,368$ & $1\,995$ & $4.7$ s &   \\
$4$ & $19$ & $5$ & $$ & $74\,113$ & $1\,495$ & $16.1$ s &   \\
$4$ & $19$ & $6$ & $$ & $4$ & $4$ & $20.1$ s &   \\
$4$ & $19$ & $7$ & $$ & $0$ & $0$ & $< 0.01$ s &   \\
\hline
$4$ & $20$ & $2$ & $$ & $21$ & $1$ &  &   \\
$4$ & $20$ & $3$ & $$ & $30\,334$ & $183$ & $16.0$ s &   \\
$4$ & $20$ & $4$ & $$ & $6\,409\,721$ & $65\,517$ & $2265.4$ s &   \\
$4$ & $20$ & $5$ & $$ & $57\,544\,941$ & $214\,717$ & $12.1$ h &   \\
$4$ & $20$ & $6$ & $$ & $25\,760$ & $745$ & $32.6$ h &   \\
$4$ & $20$ & $7$ & $$ & $0$ & $0$ &  &   \\
\hline
$4$ & $21$ & $2$ & $$ & $47$ & $3$ &  &   \\
$4$ & $21$ & $3$ & $$ & $$ & $25\,763$ & $12.1$ s & \textdagger  \\
$4$ & $21$ & $4$ & $$ & $$ & $246\,546\,229$ & $21.7$ h & \textdagger  \\
$4$ & $21$ & $5$ & $$ & $$ & $19\,419\,386\,435$ & $228.0$ d & \textdagger  \\
$4$ & $21$ & $6$ & $$ & $$ & $3\,100\,200\,221$ & $2326.4$ d & \textdagger  \\
$4$ & $21$ & $7$ & $$ & $$ & $1\,005$ & $135.4$ d & \textdagger  \\
$4$ & $21$ & $8$ & $$ & $$ & $0$ & $2.9$ s & \textdagger  \\
\hline
$5$ & $26$ & $2$ & $$ & $1$ & $1$ &  &   \\
$5$ & $26$ & $3$ & $$ & $15$ & $15$ & $0.87$ s &   \\
$5$ & $26$ & $4$ & $$ & $3$ & $3$ & $0.07$ s &   \\
$5$ & $26$ & $5$ & $$ & $6$ & $6$ & $< 0.01$ s &   \\
$5$ & $26$ & $6$ & $$ & $6$ & $6$ & $0.01$ s &   \\
$5$ & $26$ & $7$ & $$ & $0$ & $0$ & $< 0.01$ s &   \\
\hline
$5$ & $27$ & $2$ & $$ & $3$ & $1$ &  &   \\
$5$ & $27$ & $3$ & $$ & $540$ & $347$ & $16.0$ s &   \\
$5$ & $27$ & $4$ & $$ & $385$ & $193$ & $7.2$ s &   \\
$5$ & $27$ & $5$ & $$ & $3\,104$ & $1\,240$ & $2.9$ s &   \\
$5$ & $27$ & $6$ & $$ & $11\,603$ & $3\,463$ & $14.1$ s &   \\
$5$ & $27$ & $7$ & $$ & $0$ & $0$ & $22.3$ s &   \\
\hline
$5$ & $28$ & $2$ & $$ & $7$ & $1$ &  &   \\
$5$ & $28$ & $3$ & $$ & $34\,318$ & $8\,042$ & $224.9$ s &   \\
$5$ & $28$ & $4$ & $$ & $263\,321$ & $70\,992$ & $894.5$ s &   \\
$5$ & $28$ & $5$ & $$ & $4\,388\,439$ & $210\,311$ & $2874.9$ s &   \\
$5$ & $28$ & $6$ & $$ & $75\,720\,344$ & $1\,455\,113$ & $12.7$ h &   \\
$5$ & $28$ & $7$ & $$ & $0$ & $0$ & $54.5$ h &   \\
\hline
$5$ & $29$ & $2$ & $$ & $21$ & $1$ &  &   \\
$5$ & $29$ & $3$ & $$ & $2\,243\,097$ & $69\,891$ & $5808.2$ s &   \\
$5$ & $29$ & $4$ & $3$ & $148\,843$ & $19\,884$ & $44.3$ h &   \\
$5$ & $29$ & $5$ & $2$ & $36\,022$ & $31\,315$ & $1565.3$ s &   \\
$5$ & $29$ & $6$ & $1$ & $120\,074$ & $119\,047$ & $403.5$ s &   \\
$5$ & $29$ & $7$ & $$ & $281$ & $258$ & $412.5$ s &   \\
$5$ & $29$ & $8$ & $$ & $0$ & $0$ & $1.0$ s &   \\
\hline
$5$ & $30$ & $2$ & $$ & $54$ & $1$ &  &   \\
$5$ & $30$ & $3$ & $$ & $$ & $78\,086$ & $3947.2$ s &   \\
$5$ & $30$ & $4$ & $$ & $$ & $3\,002\,015\,967$ & $60.8$ d &   \\
$5$ & $30$ & $5$ & $$ & $$ & $5\,501\,626\,305$ & $645.8$ d &   \\
$5$ & $30$ & $6$ & $$ & $$ & $197\,049\,834$ & $211.8$ d &   \\
$5$ & $30$ & $7$ & $$ & $$ & $18\,857$ & $165.0$ h &   \\
$5$ & $30$ & $8$ & $$ & $$ & $0$ & $72.9$ s &   \\
\hline
$6$ & $37$ & $2$ & $$ & $1$ & $1$ &  &   \\
$6$ & $37$ & $3$ & $$ & $231$ & $231$ & $6.0$ h &   \\
$6$ & $37$ & $4$ & $$ & $13$ & $13$ & $6.2$ s &   \\
$6$ & $37$ & $5$ & $$ & $0$ & $0$ & $0.10$ s &   \\
\hline
$6$ & $38$ & $2$ & $$ & $3$ & $1$ &  &   \\
$6$ & $38$ & $3$ & $$ & $30\,491$ & $21\,371$ & $156.0$ h &   \\
$6$ & $38$ & $4$ & $$ & $8\,865$ & $6\,215$ & $1074.0$ s &   \\
$6$ & $38$ & $5$ & $$ & $0$ & $0$ & $143.7$ s &   \\
\hline
$6$ & $39$ & $2$ & $$ & $7$ & $1$ &  &   \\
$6$ & $39$ & $3$ & $$ & $5\,128\,096$ & $1\,644\,791$ & $82.1$ d &   \\
$6$ & $39$ & $4$ & $$ & $48\,249\,923$ & $19\,197\,035$ & $211.7$ h &   \\
$6$ & $39$ & $5$ & $$ & $289$ & $158$ & $248.1$ h &   \\
$6$ & $39$ & $6$ & $$ & $0$ & $0$ & $4.3$ s &   \\
\hline
$6$ & $40$ & $2$ & $$ & $21$ & $1$ &  &   \\
$6$ & $40$ & $3$ & $$ & $747\,865\,015$ & $57\,025\,160$ & $362.3$ d &   \\
$6$ & $40$ & $4$ & $2$ & $471\,192\,731$ & $85\,773\,975$ & $19213.0$ d &   \\
$6$ & $40$ & $5$ & $1$ & $388$ & $128$ & $192.1$ d &   \\
$6$ & $40$ & $6$ & $$ & $0$ & $0$ & $10.9$ s &   \\
\hline
$6$ & $41$ & $2$ & $$ & $54$ & $1$ &  &   \\
$6$ & $41$ & $3$ & $$ & $$ & $581\,769\,269$ & $756.9$ d &   \\
$6$ & $41$ & $4$ & $3$ & $$ & $1\,771\,354\,037$ & $3750.0$ d & \textdagger  \\
$6$ & $41$ & $5$ & $2$ & $$ & $61\,351$ & $541.7$ d & \textdagger  \\
$6$ & $41$ & $6$ & $1$ & $$ & $16$ & $1154.5$ s & \textdagger  \\
$6$ & $41$ & $7$ & $$ & $$ & $0$ & $0.33$ s &   \\
\hline

\end{longtable}

\subsection{Double counting}

To increase confidence in the computational results, we perform a
consistency check of the results by double counting. After the search
starting from $\CA(N;k,v)$ is performed, we count
in two ways the total number of $\CA(N;k+1,v)$ that obey the
restrictions used, that is, in some cases we count only uniform arrays
and in some cases only arrays that have $\delta$ covers of size $v$ that
intersect pairwise in exactly one row.

The first way is to use the classification results to and the orbit-stabilizer theorem to obtain
\[
\sum_{C'} \frac{(k+1)! v!^{k+1}}{|\Aut{C'}|},
\]
where $(k+1)! v!^{k+1}$ is the order of the group of symmetries in that case and the sum is taken over equivalence class representatives $C'$ of that case.

The second way is to use numbers that were stored during the
search. Consider first a modified search that starts from all
$k$-column arrays instead of equivalence class representatives and
considers all possible permutations of symbols in the last column for
each partition. If the techniques for rejecting candidates of
$\mathcal{D}$ in Section~\ref{sec:comp_isomorph} would not be used, then every
$(k+1)$-column covering array would appear exactly once when adding one more column.

If the additional condition on the largest multiplicity of a symbol in the
last column is taken into account, then the proportion of $(k+1)$-column arrays equivalent to $C'$ that
enter the isomorph rejection phase is the proportion of columns in $C'$
for which the largest multiplicity of a symbol is smallest, denoted by
$\alpha(C')$. Further, in the search starting from a fixed
$k$-column array $C$, let $\beta(C, C')$ be the proportion of all
partitions equivalent to $C'$ that pass the check~\ref{enum:chi} in
Section~\ref{sec:comp_isomorph}; this can be obtained at the stage in
the search when all partitions of $C$ equivalent to $C'$ are
considered. The total count of $(k+1)$-column arrays would now be
\[
\sum_{C,C'} \frac{1}{\alpha(C') \beta(C,C')},
\]
where the sum is taken over all $k$-column arrays $C$ and all $C'$ that are extensions of $C$ and pass the check for Condition~\ref{enum:chi}.
The remaining techniques for rejecting covers in $\mathcal{D}$ described in Section~\ref{sec:comp_isomorph} do not reject any candidates that satisfy Condition~\ref{enum:chi}, so including them in the search does not change this count.

In the modified search, arrays $C'$ which differ only by a permutation
of symbols in the last column contribute the same amount in the sum,
and the searches starting from two equivalent $C$ contribute the same
amount to the sum. Let $S(C')$ be the number of ways to assign the
symbols to the last column (this equals $v!$ if no two
parts in the corresponding partition of $C'$ are equal). Further, the size of the equivalence
class of $C$ is ${k! v!^k}/{|\Aut{C}|}$. In the actual
search, the count is then obtained as
\[
\sum_{C,C'} \frac{k! v!^{k}}{|\Aut{C}|} S(C') \frac{1}{\alpha(C') \beta(C,C')},
\]
where the sum is taken over the $k$-column arrays $C$ that are used in the search, and all $C'$ that are extensions of $C$ and pass the check for Condition~\ref{enum:chi}.

\section{Discussion of results}\label{resultsSec}

\begin{table}
\begin{center}
\begin{tabular}{cccccccc}\\\hline
$v$ & $k=4$ & $k=5$ & $k=6$ & $k=7$ & $k=8$ & $k=9$ & $k=10$\\\hline
3 &  9  & $^j11^i$ & $^l12^c$ & $^c12^k$  & $^a13^c$ & $^c13^k$ & $^a14^h$\\
4 & 16  & 16 & $^l19^k$ & $^a21^k$ & $^c21$--$22^c$ & $^c21$--$22^f$ & $^c21$--$24^h$\\
5 & 25  & 25 & 25     & $^l29^k$     & $^b30$--$33^h$ & $^c30$--$35^h$ & $^c30$--$36^e$\\
6 & $^m37^k$ & $^a39^d$ &$^b41^h$ & $^c41$--$42^c$& $^c41$--$42^h$& $^c41$--$46^g$& $^c41$--$48^n$\\\hline
\end{tabular}
\end{center}
Unmarked entries: orthogonal arrays. Captions: a) This paper, preliminarily announced in \cite{NO00}, b) This paper, c) $\CAN(k,v) \leq \CAN(k+1,v)$, d) \cite{coh1}, e) L. Rouse-Lamarre, reported in \cite{col2}, f) \cite{lob}, g) \cite{mea1}, h) \cite{nur}, i) \cite{osta}, j) Applegate, reported in \cite{slo}, k) \cite{ste3}, l) \cite{ste2}, m) \cite{tarry}, n) \cite{iimjtjhag}
\caption{Values of $\CAN(k,v)$ for $4 \leq k \leq 10$ and $3 \leq v \leq 6$} \label{tab:res}
\end{table}

A summary of the current knowledge of the sizes of optimal coverings
array for small $k$ and $v$ is given in Table~\ref{tab:res}. In the table there
are captions for all bounds, except those that follow from orthogonal arrays:
$\CAN(k,v) = v^2$ when $v$ is a prime power and $k \leq v+1$. Some of the lower
bounds attributed to the current work were obtained about two decades before this
paper appears in print. Those bounds, which were announced at a conference in
2000 \cite{NO00}, are given a caption of their own to clarify priority issues.
Some of those results have later been rediscovered
\cite{coh0,col2,tor1,tor2}.  Additionally our computer search
established that a $\CA(14;11,3)$ does not exist. The fourth author found a
$\CA(15; 20,3)$ \cite{nur} so we additionally know that $\CAN(k,3) = 15$ for $11 \leq k \leq 20$.

In the process of preparing this paper we noticed that the bound
sources listed in Table~2 of \cite{nur} are not the same as those in
Table~1 but this difference is not articulated in that article.  To
the best of our knowledge in Table~2 of \cite{nur}, b refers to
\cite{ste}, c is \cite{nur} and d is \cite{ste4}.  On page 149 of
\cite{nur}, ``giving the bounds marked with d in the tables'' should
read ``giving the bounds marked with b in Table~1 and c in Table~2''.

In every case in which we
determined the size of an optimal covering array by construction, we
also determined that there exists a uniform covering array of the same
size. These results continue to support the conjecture that optimal
covering arrays can be found amongst the uniform covering arrays.

For the parameters $\CA(11;5,3)$, $\CA(12;7,3)$, $\CA(13;8,3)$,
$\CA(13;9,3)$, $\CA(19;6,4)$, and $\CA(37;4,6)$ every optimal array
is also uniform.  However for the optimal parameters, $\CA(12;6,3)$,
$\CA(14;10,3)$, $\CA(29;7,5)$, and $\CA(39;5,6)$ both uniform and
non-uniform examples exist. Finally, for the optimal parameters
$\CA(21;7,4)$ and $\CA(41;6,6)$ we know that uniform arrays exist, but we do
not know if non-uniform examples also exist.

For the four optimal parameter sets where both uniform and non-uniform arrays exist, $\CA(12;6,3)$,
$\CA(14;10,3)$, $\CA(29;7,5)$, and $\CA(39;5,6)$, the percentages of non-isomorphic arrays that are uniform are 84.61, 91.69, 91.81, and 54.67, respectively. Combined with the optimal parameter sets where every array is uniform we can see some provisional trends.  The smallest optimal parameter set for which non-uniform arrays exist is $\CA(12;6,3)$, for which $N$ is the smallest for the given $k$, but $k$ is not the maximal possible given $N$.  We guess that non-uniform arrays will be more abundant when $k$ is not maximal for a given $N$. The second potential trend is that for a fixed $v$, as $N$ and $k$ increase there are likely to be more non-uniform optimal arrays.  These are only limited observations from few data. 

Table~\ref{tab:res_u} shows the current state of knowledge for uniform
covering arrays. All entries are lower bounds, bold entries show arrays known to exist,
and underlined entries indicate a lower bound matching Theorem~\ref{main_bound_thm}.
The lower bounds from Theorem~\ref{main_bound_thm} meet six known
uniform covering arrays.

\begin{table}
\begin{center}
\begin{tabular}{@{}cccccccc}\\\hline\hline
  $v$ & $k=4$ & $k=5$ & $k=6$ & $k=7$ & $k=8$ & $k=9$ & $k=10$\\\hline
  3&\textbf{\underline{9}}& \textbf{\underline{11}}& \textbf{\underline{12}}& \textbf{\underline{12}}& {\bf 13}& {\bf 13}& {\bf 14} \\\hline
 4& \textbf{\underline{16}}& \textbf{\underline{16}}& \textbf{\underline{19}}& \textbf{\underline{21}}& 22& 22& 22 \\\hline
 5& \textbf{\underline{25}}& \textbf{\underline{25}}& \textbf{\underline{25}}& \textbf{\underline{29}}& \underline{31}& \underline{32}& \underline{32} \\\hline
 6& {\bf 37}& {\bf 39}& {\bf 41}&  42&  42& \underline{44}& \underline{45}\\\hline
\end{tabular}
\end{center}
\caption{Lower bounds on $\UCAN(k,v)$ for $4 \leq k \leq 10$ and $3 \leq v \leq 6$} \label{tab:res_u}
\end{table}

A lower bound in Table~\ref{tab:res} that is smaller than the corresponding 
lower bound in Table~\ref{tab:res_u} indicates a candidate for a covering array
that would refute Conjecture~\ref{conj:uca}.

Conjecture~\ref{conj:uca} and Corollary~\ref{v+2_cor} would imply that
$\CAN(v+2,v) \geq v^2+v-1$, which has also been conjectured in
\cite{ste}. When $v$ is a prime power, this would mean that
$\CAN(v+2,v)-\CAN(v+1,v) \geq v-1$, which is a very large jump for
only adding a single column.  In the case of $v=6$, from
Table~\ref{tab:res} we can see that the value of $\CAN(v+1,v)$ may be
close to the value of $\CAN(v+2,v)$ when $v$ is not a prime power.

One exciting possibility is that $\CAN(8,6)$ could be 41, meeting the
bound from Theorem~\ref{main_bound_thm}; $\CAN(8,6)$ is no more than
42.  This suggests that the influence of the prime power status of $v$
disappears very rapidly as $k$ increases past $v+1$.  However, none
of the $\UCA(41;6,6)$ covering arrays can be extended to a
$\CA(41;8,6)$, so if they exist, then no subarray with six columns can
be uniform.  This implies that at least three columns of a possible $\CA(41;8,6)$ must be non-uniform.  Since there are six non-uniform partitions of 41 into six parts of size at least 6 there are 57 different partition patterns if only three columns are non-uniform.  If more columns are non-uniform, the number of cases increases. This indicates that exploiting this structure in an exhaustive search may not be efficient.  Exploiting it with a metaheuristic search could be an option.

Conjecture~\ref{conj:uca} predicts that for every covering array,
there a uniform covering array with the same parameters. We have seen
examples of optimal covering arrays which are not uniform, so we know
that not every optimal covering array is uniform. But we can ask for
which parameters are all the optimal covering arrays uniform? In this
paper we found many examples, but our examples are in cases where the
number of rows is relatively small. When $v$ is a prime power it is
possible to construct a $\CA(v^2+i(v^2-v),v^i(v+1),v)$ for any $i$
using a recursive construction and starting with an
orthogonal array (see, for example, \cite{ste}). We suspect that these parameters could be good
candidates for having every optimal covering arrays be a uniform
covering array. 

A significant result from our work is that the number of known optimal
covering arrays for $v>2$ and $N > v^2$ is now 21 whereas before it
was eleven. Additionally the $\UCA(21; 7,4)$ meet the bound from
Theorem~\ref{main_bound_thm}. This is the first example of tightness
and the implied structure, when $k > v+2$. The classification results
from our searches are available at~\cite{KMNNOS}.

In this paper we only consider strength-2 covering arrays. Many of the
questions addressed in the paper may be interesting for higher
strength covering arrays. The definition of ``uniform'' applies to
covering arrays of any size and it is interesting to ask if it is
always possible to find an optimal covering array, of any strength, that is also uniform. Extending Theorem~\ref{main_bound_thm} to strength $t$ would require counting pairs of rows which agree in at most $t-1$ positions and would be an interesting investigation. For strength $t>2$, extending an array with an additional column requires determining all of the strength $t-1$ subarrays which is more computationally demanding as $t$ increase. The use of \emph{nauty} to compute automorphism groups depends mainly on the sizes of the arrays and is not inherently more complicated as the strength increases.  For classification and fully understanding the structure of optimal arrays, we do not see better options than exhaustive search.

\section*{Acknowledgment}
The authors wish to thank the referees
for useful comments that helped improve this article.


\begin{thebibliography}{XX}
\bibitem{akh0} Y. Akhtar, S. Maity, and R. C. Chandrasekharan,
  Covering arrays of strength four and software testing, in:
  R. N. Mohapatra, D. R. Chowdhury, and D. Giri (Eds.), Mathematics
  and Computing, Springer Proc. Math. Stat. 139, Springer, New Delhi,
  2015, pp.~391--398.

\bibitem{coh0} M.B. Cohen, personal communication 2014.

\bibitem{coh1} M.B. Cohen, Designing test suites for software
  interaction testing, Ph.D. Thesis, University of Auckland, 2004.

\bibitem{col0} C. J. Colbourn, Combinatorial aspects of covering
  arrays, Le Matematiche (Catania), {\bf 59} (2006), 125--172.

\bibitem{col3} C. J. Colbourn, Augmentation of covering arrays of
  strength two, Graphs Combin. {\bf 31} (2015), 2137--2147.

\bibitem{mr3440525} C. J. Colbourn, Suitable permutations, binary
  covering arrays, and Paley matrices, Springer Proc. Math. Stat. {\bf
    133} (2015), 29--42.


\bibitem{col2} C. J. Colbourn, G. K\'{e}ri, P. P. R. Soriano, and
  J.-C. Schlage-Puchta, Covering and radius-covering arrays:
  constructions and classification, Discrete Appl. Math. {\bf 158}
  (2010), 1158--1180.

\bibitem{colbourn_asymptotic_nodate} C. J. Colbourn, K. Sarkar, and
  E. Lanus, Asymptotic and constructive methods for covering perfect
  hash families and covering arrays, Des. Codes Cryptogr. {\bf 86}
  (2018), 907--937.

\bibitem{danziger_covering_2009} P. Danziger, E. Mendelsohn, L. Moura,
  and B. Stevens, Covering arrays avoiding forbidden edges,
  Theoret. Comput. Sci. {\bf 410} (2009), 5403--5414.

\bibitem{francetic2016} N. Franceti\'{c} and B. Stevens, Asymptotic
  size of covering arrays: an application of entropy compression,
  J. Combin. Des. {\bf 25} (2017), 243--257.

\bibitem{hart0} A. Hartman, Software and hardware testing using
  combinatorial covering suites, in: M. C. Golumbic and
  I. B.-A. Hartman (Eds.), Graph Theory, Combinatorics and Algorithms,
  Springer, New York, 2005, pp.~237--266.

\bibitem{MR2224851} B. Hnich, S. D. Prestwich, E. Selensky,
  and B. M. Smith, {Constraint models for the covering test
    problem}, Constraints {\bf 11} (2006), 199--219.

\bibitem{iimjtjhag} I. Izquierdo-Marquez, J. Torres-Jimenez, and
  H. Avila-George, New upper bounds for pairwise senary test-suites,
  submitted for publication.

\bibitem{kas} P. Kaski and P. R. J. {\"O}sterg{\aa}rd, Classification
  Algorithms for Codes and Designs, Springer, Berlin, 2006.

\bibitem{kas2} P. Kaski and O. Pottonen, libexact user's guide,
  Version 1.0, Helsinki Institute for Information Technology HIIT,
  Helsinki, 2008.

\bibitem{katona} G. O. H. Katona, Two applications (for search theory
  and truth functions) of Sperner type theorems,
  Period. Math. Hungar. 3 (1973), 19--26.

\bibitem{kleitmanspencer} D. J. Kleitman and J. Spencer, Families of
  $k$-independent sets, Discrete Math. {\bf 6} (1973), 255--262.

\bibitem{KMNNOS} J. I. Kokkala, K. Meagher, R. Naserasr,
  K. J. Nurmela, P. R. J. \"{O}sterg{\aa}rd, and B. Stevens, Dataset
  for On the structure of small strength-2 covering arrays
  [Dataset]. Zenodo.  https://doi.org/10.5281/zenodo.1476059 (October
  31, 2018).

\bibitem{law0} J. Lawrence, R. N. Kacker, Y. Lei, D. R. Kuhn, and
  M. Forbes, A survey of binary covering arrays,
  Electron. J. Combin. {\bf 18} (2011), P84.

\bibitem{lob} J. R. Lobb, C. J. Colbourn, P. Danziger, B. Stevens, and
  J. Torres-Jimenez, Cover starters for covering arrays of strength
  two, Discrete Math. {\bf 312} (2012), 943--956.

\bibitem{mck2} B. D. McKay, Isomorph-free exhaustive generation,
  J. Algorithms {\bf 26} (1998), 306--324.

\bibitem{mck_pip} B. D. McKay and A. Piperno, Practical graph
  isomorphism, II, J. Symbolic Comput. {\bf 60} (2014), 94--112.

\bibitem{mea0} K. Meagher and B. Stevens, Covering arrays on graphs,
  J. Combin. Theory Ser. B {\bf 95} (2005), 134--151.

\bibitem{mea1} K. Meagher and B. Stevens, Group construction of
  covering arrays, J. Combin. Des. {\bf 13} (2005), 70--77.


\bibitem{MR2974273} P. Nayeri, C. J. Colbourn, and G.
  Konjevod, Randomized post-optimization of covering arrays, European
  J. Combin. {\bf 34} (2013), 91--103.

\bibitem{NO03} S.~Niskanen, P.~R.~J. {\"{O}}sterg{\aa}rd, Cliquer
  User's Guide, Version 1.0, Tech. Rep. T48, Communications
  Laboratory, Helsinki University of Technology, Espoo, 2003.

\bibitem{nur} K. J. Nurmela, Upper bounds for covering arrays by tabu
  search, Discrete Appl. Math. {\bf 138} (2004), 143--152.

\bibitem{NO00} K. J. Nurmela and P. R. J. \"Osterg{\aa}rd, Lower
  bounds on 2-covering arrays by exhaustive search, presented at the
  \emph{25th Australasian Conference on Combinatorial Mathematics and
    Combinatorial Computing} (Christchurch, New Zealand, December
  4--8, 2000).

\bibitem{osta} P. R. J. \"Osterg{\aa}rd, Constructions of mixed
  covering codes, Research Report A18, Digital Systems Laboratory,
  Helsinki University of Technology, Espoo, 1991.

\bibitem{ost0} P. R. J. \"Osterg{\aa}rd, Disproof of a conjecture on
  the existence of balanced optimal covering codes, IEEE
  Trans. Inform. Theory {\bf 49} (2003), 487--488.

\bibitem{ost} P. R. J. \"Osterg\aa rd, On optimal binary codes with
  unbalanced coordinates, Appl. Algebra Engrg. Comm. Comput. {\bf 24}
  (2013), 197--200.

\bibitem{OBK99} P.~R.~J. \"{O}sterg{\aa}rd, T.~Baicheva, and E.~Kolev,
  Optimal binary one-error-correcting codes of length 10 have 72
  codewords, IEEE Trans.  Inform. Theory {\bf 45} (1999), 1229--1231.

\bibitem{sarkar_upper_2017} K. Sarkar and C. J Colbourn, Upper bounds
  on the size of covering arrays, SIAM J. Discrete Math. {\bf 31}
  (2017), 1277--1293.

\bibitem{mr3565209} K. Sarkar, C. J. Colbourn, A. de Bonis, and
  U. Vaccaro, Partial covering arrays: algorithms and asymptotics, in:
  V. M\"{a}kinen, S. J. Puglisi and L. Salmela (Eds.), Combinatorial
  Algorithms, LNCS 9843, Springer, Cham, 2016, pp. 437--448.

\bibitem{slo} N. J. A. Sloane, Covering arrays and intersecting codes,
  J. Combin. Des. \textbf{1} (1993), 51--63.

\bibitem{ste} B. Stevens, Transversal Covers and Packings,
  Ph.D. Thesis, University of Toronto, Toronto, 1998.

\bibitem{set_nonu} B. Stevens, Non-uniform covering array with
  symmetric forbidden edge constraints, preprint available at
  \url{https://arxiv.org/abs/1901.02479}, submitted for publication.

\bibitem{ste4} B. Stevens, A. Ling and E. Mendelsohn, A direct
  construction of transversal covers using group divisible designs,
  Ars Combin. {\bf 63} (2002), 145--159.

\bibitem{ste3} B. Stevens and E. Mendelsohn, New recursive methods for
  transversal covers, J. Combin. Des. {\bf 7} (1999), 185--203.

\bibitem{ste2} B. Stevens, L. Moura, and E. Mendelsohn, Lower bounds
  for transversal covers, Des. Codes Cryptogr. {\bf 15} (1999),
  279--299.
 
\bibitem{tarry} G. Tarry, Le Probl\`{e}m des 36 Officiers,
  C. R. Assoc. Fr. Av. Sci. {\bf 29}(2) (1900), 170--203.

\bibitem{tor1} J. Torres-Jimenez, personal communication, 2016.

\bibitem{tor2} J. Torres-Jimenez and I. Izquierdo-Marquez,
  Construction of non-isomorphic covering arrays, Discrete
  Math. Algorithms Appl. \textbf{8} (2016), 1650033.

\bibitem{tor3} J. Torres-Jimenez and E. Rodriguez-Tello, New bounds
  for binary covering arrays using simulated annealing,
  Inform. Sci. {\bf 185} (2012), 137--152.

\bibitem{tza1} G. Tzanakis, L. Moura, D. Panario, and B. Stevens,
  Covering arrays from \mbox{m-sequences} and character sums, Des. Codes
  Cryptogr. {\bf 85} (2017), 437--456.

\bibitem{MR3328867} R. Yuan, Z. Koch, and A. Godbole, Covering array
  bounds using analytical techniques, Congr. Numer. {\bf 222} (2014),
  65--73.

\end{thebibliography}
\end{document}